\documentclass[10pt]{amsart}
\usepackage{verbatim}
\usepackage{eucal,url,amssymb,stmaryrd,enumerate,amscd,verbatim}
\usepackage{amsmath}
\usepackage[pagebackref,colorlinks=true,linkcolor=blue,citecolor=blue]{hyperref}
\usepackage{amsfonts}
\usepackage{amsthm}
\usepackage[margin=1in]{geometry}

\allowdisplaybreaks
\numberwithin{equation}{section}
\newtheorem{thrm}{Theorem}[section]
\newtheorem{lemma}[thrm]{Lemma}
\newtheorem{prop}[thrm]{Proposition}
\newtheorem{cor}[thrm]{Corollary}

\newtheorem{rmrk}[thrm]{Remark}

\newtheorem{conv}[thrm]{Convention}

 1

\def\gr{\nabla f}
\def\g{\nabla\varphi}
\def\bi{\nabla}

\newcommand{\vol}{\, Vol_{\eta}}
\newcommand{\Vol}{\, Vol_{\theta}}

\begin{document}

\begin{abstract}
We establish a new version of the CR almost Schur Lemma which gives an estimation of the pseudohermitian scalar curvature on a compact strictly pseudoconvex pseudohermitian manifold to be a constant in terms of the norm of the traceless Webster Ricci tensor and the pseudohermitian torsion  under a certain positivity condition. In the torsion-free case, i.e. for a compact Sasakian manifold, our positivity condition coincides with the known one and we obtain a better estimate.
\end{abstract}

\title[The CR Almost Schur Lemma and the positivity conditions]{The CR Almost Schur Lemma and the positivity conditions}
\date{\today }
\author{Stefan Ivanov}
\address[Stefan Ivanov]{University of Sofia, Faculty of Mathematics and
Informatics, blvd. James Bourchier 5, 1164, Sofia, Bulgaria}
\address{and Institute of Mathematics and Informatics, Bulgarian Academy of
Sciences} \email{ivanovsp@fmi.uni-sofia.bg}
\author{Alexander Petkov}
\address[Alexander Petkov]{University of Sofia, Faculty of Mathematics and
Informatics, blvd. James Bourchier 5, 1164, Sofia, Bulgaria}
\email{a\_petkov\_fmi@abv.bg}

\maketitle

\tableofcontents


\setcounter{tocdepth}{2}

\section{Introduction}
For a compact Riemannian manifold $(M^n,g)$ of dimension $n\geq 3$ the famous Schur lemma states that if $(M^n,g)$ is Einstein, then it has constant scalar curvature, $S=Const.$ The metric $g$ is said to be Einstein, if the Ricci tensor is proportional to the metric, $Ric=\frac{S}{n}g.$ A generalization of the Schur lemma is the  result of De Lellis and Topping \cite{LT} that states as follows.
\begin{thrm}\cite[Almost Schur Lemma]{LT} Let $(M^n,g)$ be a compact Riemannian manifold of dimension $n\geq 3$ with non--negative Ricci tensor, $Ric\geq 0$. Then the following inequality holds
\begin{equation*}\label{ASL}
\int_M(S-\bar{S})^2 vol_g\leq\frac{4n(n-1)}{(n-2)^2}\int_M\left|Ric-\frac{S}{n}g\right|_g^2 vol_g,
\end{equation*}
where $\bar{S}$ means the average value of the scalar curvature $S$ of $g$.

The equality holds if and only if the manifold is Einstein.
\end{thrm}
It is also shown in \cite{LT} that the positivity condition $Ric\geq 0$ assumed on the Ricci tensor is essential and can not be dropped. 

In the CR case there are known two positivity conditions written in terms of the Webster Ricci curvature and the pseudohermitian torsion; one is used for obtaining a lower bound of the first eigenvalue of the sub-Laplacian (see e.g. \cite{Gr}), while the other one appears in the CR  Cordes type estimate (see \cite{CM}). 

A CR  version of the almost Schur lemma was first established in \cite{CSW} in terms of the Tanaka-Webster connection, its Ricci curvature and the Webster torsion. The key assumption is the  positivity condition \eqref{lich} below. We present  this result in real notations.
\begin{thrm}\cite[Theorem~1.2]{CSW}\label{thcw}
Let $n\ge 2$ and $(M,J,\theta)$ be a $(2n+1)$-dimensional compact strictly pseudoconvex pseudohermitian manifold satisfying the condition
\begin{equation}\label{lich}
Ric(X,X)+4A(JX,X)=Rc(X,X)+2(n+1)A(JX,X)\ge 0, X\in H.
\end{equation}
Then the following inequality holds
\begin{multline}\label{s2fb}
\int_M(S-\bar S)^2\Vol
\le\frac{4n(n+1)}{(n-1)(n+2)}\int_M|Rc_0|^2\Vol-8n\int_M\sum_{a,b=1}^{2n}A(e_a,Je_b)(\bi^2\varphi)_{[-1]}(e_a,e_b)\Vol.
\end{multline}
If the equality holds, then 
\begin{equation}\label{seqb}
\int_M(S-\bar S)^2\Vol
=\frac{4n(n+1)}{(n-1)(n+2)}\int_M|Rc_0|^2\Vol
\end{equation}
and the manifold is CR equivalent  to a pseudo-Einstein space.
\end{thrm}

In the inequalities \eqref{lich} and \eqref{s2fb},  $Ric$ is the  Ricci tensor of the Tanaka-Webster connection, $Rc$ is the Webster Ricci tensor,  $A$ is the pseudohermitian torsion, $S$ and $\overline{S}=\int_MS\Vol$ are the  pseudohermitian scalar curvature and its average value,  respectively, and $Rc_0=Rc-\frac{S}{2n}g,\quad |Rc_0|^2=\sum_{a,b=1}^{2n}Rc_0(e_a,e_b)Rc_0(e_a,e_b)$ are  the trace-free part of the Webster Ricci tensor and its  horizontal norm, respectively.
Finally, $\varphi$ is the unique solution of the sub-elliptic equation $$\Delta\varphi=S-\overline{S}\quad\textnormal{with} \quad \int_M\varphi\Vol=0.$$ 

The aim of this note  is to present  another version of the CR almost Schur lemma with a different  positivity condition. We assume the positivity condition  \eqref{cor} below  instead of \eqref{lich}.  In the torsion-free case, i.e. for  Sasakian manifolds,   both positivity conditions coincide and we  obtain a better estimate then that following from Theorem~\ref{thcw}. Our main result is 
\begin{thrm}\label{main}
Let $n\ge 2$ and $(M,J,\theta)$ be a $(2n+1)$-dimensional compact strictly pseudoconvex pseudohermitian manifold satisfying the condition
\begin{equation}\label{cor}
Ric(X,X)+6A(JX,X)=Rc(X,X)+2(n+2)A(JX,X)\ge 0, X\in H.
\end{equation}
Then the following inequality holds
\begin{multline}\label{s2fa}
\int_M(S-\bar S)^2\Vol
\le\frac{2n(2n+3)}{(n-1)(n+3)}\int_M|Rc_0|^2\Vol-8n\int_M\sum_{a,b=1}^{2n}A(e_a,Je_b)(\bi^2\varphi)_{[-1]}(e_a,e_b)\Vol.
\end{multline}
If the equality holds, then 
\begin{equation}\label{seqa}
\int_M(S-\bar S)^2\Vol
=\frac{2n(2n+3)}{(n-1)(n+3)}\int_M|Rc_0|^2\Vol
\end{equation}
and the manifold is CR equivalent  to a pseudo-Einstein space.
\end{thrm}
\begin{rmrk}
Note that the expression in the  left-hand side of \eqref{lich}	is precisely the CR-Lichnerowicz condition used to find a lower bound of the first eigenvalue of the sub-Laplacian  (see \cite{Gr,LL,Chi06,CC09a,CC09b,IVO}), while the expression in the left-hand side of\eqref{cor}  appears in the CR  Cordes  type a priori inequality between the (horizontal) Hessian and the sub-Laplacian of a function, derived in \cite[Theorem~1]{CM}, see Theorem~\ref{cord} below.
\end{rmrk}
The second Bianchi identity \eqref{bi2w} below shows that a compact pseudo-Einstein space has constant pseudohermitian scalar curvature if and only if the next condition holds
\begin{equation}\label{co1}
(\nabla_{e_b}\nabla_{e_a}A)(e_a,Je_b)=0,
\end{equation}
and it seems natural to assume the condition \eqref{co1} in order to have constant  pseudohermitian scalar curvature.
\begin{cor}\label{maincor}
If, in addition to the conditions of Theorem~\ref{main}, we suppose the equality \eqref{co1} holds,
then
\begin{equation}\label{s2fab0}
\int_M(S-\bar S)^2\Vol
\le\frac{2n(2n+3)}{(n-1)(n+3)}\int_M|Rc_0|^2\Vol.
\end{equation}
If we have  equality in \eqref{s2fab0}  then  the compact pseudohermitian manifold is pseudo-Einstein with constant pseudohermitian scalar curvature.
\end{cor}
The next result slightly improves \cite[Corollary~1.3]{CSW}.
\begin{cor}\label{im}
If, in addition to the conditions of Theorem~\ref{thcw}, we suppose the equality \eqref{co1} holds, 
then
\begin{equation}\label{seqbc}
\int_M(S-\bar S)^2\Vol\leq\frac{4n(n+1)}{(n-1)(n+2)}\int_M|Rc_0|^2\Vol.
\end{equation}
If we have equality in \eqref{seqbc} then the compact pseudohermitian manifold is pseudo-Einstein with constant pseudohermitian scalar curvature.
\end{cor}
In the  torsion-free case we get from Corollary~\ref{maincor}
\begin{cor}\label{mainsas}
Let $n\ge 2$ and $(M,J,\theta)$ be a $(2n+1)$-dimensional compact torsion-free strictly pseudoconvex pseudohermitian manifold, i.e. a Sasakian manifold, with non-negative Webster Ricci tensor, $Rc\ge 0$. 
Then the following inequality holds
\begin{equation}\label{s2fab}
\int_M(S-\bar S)^2\Vol
\le\frac{2n(2n+3)}{(n-1)(n+3)}\int_M|Rc_0|^2\Vol.
\end{equation}
If the equality in \eqref{s2fab} holds, then    the compact Sasakian space is pseudo-Einstein with constant pseudohermitian  scalar curvature  and therefore it is a Riemannian Sasaki $\eta$-Einstein space  D-homothetic to a Riemannian Sasaki-Einstein space.



\end{cor}

\begin{rmrk}
In the torsion-free case, $A=0$, the positivity assumptions \eqref{lich} and \eqref{cor} coincide and because  the number $\frac{2n(2n+3)}{(n-1)(n+3)}$ is smaller than the number $\frac{4n(n+1)}{(n-1)(n+2)}$,
 we get a better estimate for Sasakian manifolds than the one following from \cite[Theorem~1.2]{CSW},  i.e. Theorem~\ref{thcw} above.
\end{rmrk}
In the proof of  Crollary~\ref{im} we also present a proof of Theorem~\ref{thcw}, \cite[Theorem~1.2]{CSW}. 

In the Appendix we  record  for self-sufficiency  some of the results of \cite{GL88,Gr,L1} in real variables (see also \cite[Appendix]{IVO}) including the Greenleaf's CR Bochner formula \cite{Gr}, the CR Paneitz operator and its non-negativity for $n>1$ \cite{GL88}.

\begin{conv}
\label{conven} \hfill\break\vspace{-15pt}

\begin{enumerate}[ a)]

\item We shall use $X,Y,Z,U$ to denote horizontal vector fields, i.e. $X,Y,Z,U\in H$.

\item $\{e_1,\dots,e_{2n}\}$ denotes a local orthonormal basis of the
horizontal space $H$.

\item The summation convention over repeated vectors from the basis $\{e_1,\dots,e_{2n}\}$ will be used. 
For example, for a (0,4)-tensor $P$, the
formula $k=P(e_b,e_a,e_a,e_b)$ means $ k=\sum_{a,b=1}^{2n}P(e_b,e_a,e_a,e_b).$


\end{enumerate}
\end{conv}

\textbf{Acknowledgments}  The research of both authors  is partially supported   by Contract DH/12/3/12.\allowbreak{}12.2017,
Contract 80-10-161/05.04.2021  with the Sofia University "St. Kliment Ohridski", 
and
the National Science Fund of Bulgaria, National Scientific Program "VIHREN'', Project No. KP-06-DV-7.

\section{Pseudohermitian manifolds and the Tanaka-Webster connection}

In this section we will briefly review the basic notions of the
pseudohermitian geometry of a CR manifold. Also, we recall some results (in
their real form) from \cite{L1,T,W,W1}, see also \cite{DT,IVO,IVZ}, which we
will use in this paper.

A CR manifold is a smooth manifold $M$ of real dimension 2n+1, with a fixed
n-dimensional complex sub-bundle $\mathcal{H}$ of the complexified tangent
bundle $\mathbb{C}TM$ satisfying $\mathcal{H} \cap \overline{\mathcal{H}}=0$
and $[ \mathcal{H},\mathcal{H}]\subset \mathcal{H}$. If we let $H=Re\,
\mathcal{H}\oplus\overline{\mathcal{H}}$, the real sub-bundle $H$ is
equipped with a formally integrable almost complex structure $J$. We assume
that $M$ is oriented and there exists a globally defined compatible contact
form $\theta$ such that ${H}=Ker\,\theta$. In other words, the hermitian
bilinear form
\begin{equation*}
2g(X,Y)=-d\theta(JX,Y)
\end{equation*}
is non-degenerate. The CR structure is called strictly pseudoconvex if $g$
is a positive definite tensor on $H$. The vector field $\xi$ dual to $\theta$
with respect to $g $ satisfying $\xi\lrcorner d\theta=0$ is called the Reeb
vector field. The almost complex structure $J$ is formally integrable in the
sense that
\begin{equation*}
([JX,Y]+[X,JY])\in {H}
\end{equation*}
and the Nijenhuis tensor
\begin{equation*}
N^J(X,Y)=[JX,JY]-[X,Y]-J[JX,Y]-J[X,JY]=0.
\end{equation*}
A CR manifold $(M,\theta,g)$ with a fixed compatible contact form $\theta$
is called \emph{a pseudohermitian manifold}%
\index{pseudohermitian manifold}. In this case the 2-form
\begin{equation*}
d\theta_{|_{{H}}}:=2\omega
\end{equation*}
is called the fundamental form. Note that the contact form is determined up
to a conformal factor, i.e. $\bar\theta=\nu\theta$ for a positive smooth
function $\nu$ defines another pseudohermitian structure called
pseudo-conformal to the original one.

\subsection{Invariant decompositions}

As usual any endomorphism $\Psi$ of $H$ can be decomposed with respect to
the complex structure $J$ uniquely into its $U(n)$-invariant $(2,0)+(0,2)$
and $(1,1)$ parts. In short we will denote these components correspondingly
by $\Psi_{[-1]}$ and $\Psi_{[1]}$. Furthermore, we shall use the same
notation for the corresponding  $2$-tensor, $\Psi(X,Y)=g(\Psi X,Y)$.
Explicitly, $\Psi=\Psi_{[1]}+\Psi_{[-1]}$, where
\begin{equation}
{\label{comp}} \Psi_{[1]}(X,Y)=%
\frac {1}{2}\left [ \Psi(X,Y)+\Psi(JX,JY)\right ], \qquad \Psi_{[-1]}(X,Y)=%
\frac {1}{2}\left [ \Psi(X,Y)-\Psi(JX,JY)\right ].
\end{equation}
The above notation is justified by the fact that the $(2,0)+(0,2)$ and $%
(1,1) $ components are the projections on the eigenspaces of the operator
\begin{equation*} 
\Upsilon =\ J\otimes J, \quad (\Upsilon \Psi) (X,Y)\overset{def}{=}\Psi
(JX,JY),
\end{equation*}
corresponding, respectively, to the eigenvalues $-1$ and $1$. Note that both
the metric $g$ and the 2-form $\omega$ belong to the [1]-component, since $%
g(X,Y)=g (JX,JY)$ and $\omega(X,Y)=\omega (JX,JY)$. Furthermore, the two
components are orthogonal to each other with respect to $g$.

\subsection{The Tanaka-Webster connection}

The Tanaka-Webster connection \cite{T,W,W1} is the unique linear connection $%
\nabla$ with torsion $T$ preserving a given pseudohermitian structure, i.e.,
it has the properties
\begin{equation}  \label{crcant2}
\nabla\xi=\nabla J=\nabla\theta=\nabla g=0,
\end{equation}
\begin{equation}  \label{torha}
\begin{aligned} & T(X,Y)=d\theta(X,Y)\xi=2\omega(X,Y)\xi, \quad T(\xi,X)\in
{H}, \\ & g(T(\xi,X),Y)=g(T(\xi,Y),X)=-g(T(\xi,JX),JY). \end{aligned}
\end{equation}
It is well known that the endomorphism $T(\xi,\cdot)$ is the obstruction for a
pseudohermitian manifold to be Sasakian. The symmetric endomorphism $T_\xi:{H%
}\longrightarrow {H}$ is denoted by $A$, $A(X,Y):=T(\xi,X,Y)$,  and it is called \emph{the torsion of
the pseudohermitian manifold} or \emph{pseudohermitian torsion.} The
pseudohermitian torsion $A$ is completely trace-free satisfying
\begin{equation} \label{tortrace}
A(e_a,e_a)=A(e_a,Je_a)=0, \quad A(X,Y)=A(Y,X)=-A(JX,JY).
\end{equation}
Let $R$ be the curvature of the Tanaka-Webster connection. The
pseudohermitian Ricci tensor $Ric$, the pseudohermitian scalar curvature $S$
and the pseudohermitian Ricci 2-form $\rho$ are defined by
\begin{equation*}
Ric(A,B)=R(e_a,A,B,e_a), \quad S=Ric(e_a,e_a),\quad
\rho(A,B)=\frac12R(A,B,e_a,Je_a).
\end{equation*}
We summarize below the well known properties of the curvature $R$ of the
Tanaka-Webster connection \cite{W,W1,L1} using real expression, see also
\cite{DT,IVZ,IVO, IV2}.
\begin{gather}  \label{torric}
Ric(X,Y)=Ric(Y,X),\quad
Ric(X,Y)-Ric(JX,JY)=4(n-1)A(X,JY),\\
\label{rho}
2\rho(X,JY)=-Ric(X,Y)-Ric(JX,JY)=R(e_a,Je_a,X,JY),\\
\label{div}
2(\nabla_{e_a}Ric)(e_a,X)= dS(X).\hskip2.5truein
\end{gather}
The equalities \eqref{torric} and \eqref{rho} imply
\begin{equation}  \label{rid}
Ric(X,Y)=\rho(JX,Y)+2(n-1)A(JX,Y),
\end{equation}
i.e. $\rho$ is the $(1,1)$-part of the pseudohermitian Ricci tensor, while
the $(2,0)+(0,2)$-part is given by the pseudohermitian torsion $A$. 

The Webster Ricci tensor $Rc$ is defined by
\begin{equation}\label{wric}
Rc(X,Y)=\rho(JX,Y)=Rc(JX,JY)=Rc(Y,X).
\end{equation}
The Webster Ricci tensor $Rc$ is symmetric, of type (1,1) with respect to $J$ and shares the same trace with the Ricci tensor $Ric$, $S=Ric(e_a,e_a)=\rho(Je_a,e_a)=Rc(e_a,e_a)$ due to \eqref{rid}.

The trace-free part $Rc_0$ of the Webster Ricci tensor is given by 
\begin{equation}\label{wric0}
Rc_0(X,Y)=Rc(X,Y)-\frac{S}{2n}g(X,Y).
\end{equation}
We recall that a pseudohermitian manifold is called pseudo-Einstein if the trace-free part of the Webster Ricci tensor vanishes.

\subsection{The Ricci identities for the Tanaka-Webster connection}

Let $f$ be a smooth function on a pseudohermitian manifold $M$ with $\nabla
f $ its horizontal gradient, $g(\nabla f,X)=df(X)$. 
The sub-Laplacian (or horizontal Laplacian) $\triangle f$ and the norm of the horizontal
gradient $\nabla f =df(e_a)e_a$ of a smooth function $f$ on $M$ are defined
respectively by
\begin{equation}  \label{lap}
\triangle f\ =-\ tr^g_H(\nabla^2f)\ = -\ \nabla^2f(e_a,e_a),
\qquad |\nabla f|^2\ =\ df(e_a)\,df(e_a).
\end{equation}

We have the next
Ricci identities for the Tanaka-Webster connection following from the general Ricci identities for a connection with torsion applying the properties of the pseudohermitian torsion listed in \eqref{torha}: 
\begin{equation}  \label{e:ricci identities}
\begin{aligned} & \nabla^2f (X,Y)-\nabla^2f(Y,X)=-2\omega(X,Y)df(\xi)\\ &
\nabla^2f (X,\xi)-\nabla^2f(\xi,X)=A(X,\nabla f)\\ & \nabla^3 f
(X,Y,Z)-\nabla^3 f(Y,X,Z)=-R(X,Y,Z,\nabla f) - 2\omega(X,Y)\nabla^2f
(\xi,Z)\\ &\nabla^3 f(X,Y,Z)-\nabla^3f(Z,Y,X)=-R(X,Y,Z,\nabla f)-R(Y,Z,X,\nabla f)-2\omega(X,Y)\nabla^2f(\xi,Z)\\ &-2\omega(Y,Z)\nabla^2f(\xi,X)+2\omega(Z,X)\nabla^2f(\xi,Y)+2\omega(Z,X)A(Y,\nabla f).
\end{aligned}
\end{equation}

An important consequence of  the first Ricci identity is the following fundamental formula
\begin{equation}  \label{xi1}
g(\nabla^2f,\omega)=\nabla^2f(e_a,Je_a)=-2n\,df(\xi).
\end{equation}
On the other hand, it follows from \eqref{lap} that the trace with respect to the metric is the sub-Laplacian:
\[
g(\nabla^2f,g)=\nabla^2f(e_a,e_a)=-\triangle f.
\]

We also recall the horizontal divergence theorem \cite{T}. Let $(M,
g,\theta) $ be a pseudohermitian manifold of dimension $2n+1$. For a fixed
local 1-form $\theta$ the form
$Vol_{\theta}=\theta\wedge\omega^{n}$ is a globally defined volume form since $Vol_{\theta}$ is independent on the
local 1-form $\theta$. The (horizontal) divergence of a horizontal vector
field/one-form $\sigma\in\Lambda^1\, (H)$ is defined by 
$
\nabla^*\, \sigma\ =-tr|_{H}\nabla\sigma=\ -(\nabla_{e_a}\sigma)e_a.
$ 
 If the manifold is compact it is well known  \cite{T} that
\begin{equation*}
\int_M (\nabla^*\sigma)Vol_{\theta}\ =\ 0.
\end{equation*}

{
\section{The positivity conditions, CR Cordes type inequality}
In this section we discuss the importance of the positivity conditions \eqref{lich} and \eqref{cor}, recover the CR Cordes type inequality esteblished in \cite{CM} giving more information in the equality case.

We recall from \cite{L1}, (see also \cite{BF74} and \cite{Be80}),  that if $%
n\geq 2$, a function $f\in \mathcal{C}^3(M)$ on a compact $(2n+1)$- dimensional strictly pseudoconvex pseudohermitian manifold is CR-pluriharmonic, i.e.,
locally it is the real part of a CR holomorphic function, if and only if  the non-negative CR Paneitz operator of $f$ vanishes, i.e.  the (1,1) trace-free part of the horizontal Hessian of $f$ is zero, $ 
(\nabla^2f)_{[1][0]}=0$. In fact, as shown in \cite{GL88}, 
only one fourth-order equation $Cf=0$ suffices for $(\nabla^2f)_{[1][0]}=0$ to hold. 

Now,  we recover the CR Cordes estimate  proved in \cite[Theorem~1]{CM} and give a bit more information in the equality case for dimensions bigger than three.
\begin{thrm}\cite{CM}\label{cord}
On a compact strictly pseudoconvex pseudohermitian manifold of dimension $2n+1$,  for any real-valued function $f$ and $n>1$ we have the inequality
\begin{multline}  \label{bohinw1}
\frac{n+2}{n}\int_M(\triangle f)^2\Vol\ge\int_M\Big[Rc(\nabla f,\nabla f)+2(n+2)A(J\nabla f,\nabla f)  +\left |(\nabla^2f)_{[1]}\right |^2 + \left
|(\nabla^2f)_{[-1]}\right |^2\Big]\Vol.
\end{multline}
The equality is achieved only for CR-pluriharmonic functions.

If, moreover, the positivity condition \eqref{cor} holds then
$$\frac{n+2}{n}\int_M(\triangle f)^2\Vol\ge\int_M\Big[\left |(\nabla^2f)_{[1]}\right |^2 +\left | (\nabla^2f)_{[-1]}\right |^2\Big]\Vol.
$$

In the case $n=1$, if we assume in addition that the CR-Paneitz operator is non-negative, then the inequality \eqref{bohinw1} holds also for $n=1$.
\end{thrm}
\begin{proof}
Insert \eqref{idcr2} into \eqref{bohin1} and apply \eqref{rid} and \eqref{wric} to obtain
\begin{multline}  \label{e:bohinw1}
\frac{n+2}{n}\int_M(\triangle f)^2\Vol=\int_M\Big[Rc(\nabla f,\nabla f)+2(n+2)A(J\nabla f,\nabla f)  \Big]\Vol \\+\int_M\Big[\left |(\nabla^2f)_{[1]}\right |^2 + \left
|(\nabla^2f)_{[-1]}\right |^2-\frac {2}{n}P(\gr)\Big]\Vol.
\end{multline}
Due to Lemma~\ref{l:GrLee} the last term in \eqref{e:bohinw1} is non-negative, $-P(\bi f)\ge 0$, and the inequality \eqref{bohinw1} follows from \eqref{e:bohinw1}.  Moreover, the equality in \eqref{bohinw1} can be achieved only when $P(\bi f)= 0=(\bi f)^2_{[1][0]}$  due again to Lemma~\ref{l:GrLee}. The assertion follows from the discussion above the formulation of the theorem.
\end{proof}
\begin{rmrk}
It was pointed out in \cite{CC09b} that in the original Greenleaf's CR-Bochner formula from \cite{Gr} the coefficient in front of $A$ is not correct. In the proof of \cite[Theorem~1]{CM} it was used  that incorrect CR-Bochner formula  and therefore the coefficient in front of the pseudohermitian  torsion $A$ in  \eqref{bohinw1}  differs a bit from the  coefficient in front of $A$ in the   formula  (22) of \cite[Theorem~1]{CM}.
\end{rmrk}
Applying \eqref{Aa} to the identity \eqref{e:bohinw1}, taking into account \eqref{hes10}, we also get 
\begin{multline}\label{e:bohinw2}
\frac {n+1}{n}\int_M(\triangle f)^2\Vol=\int_M\Big[Rc(\nabla f,\nabla f)+2(n+1)A(J\nabla f,\nabla f)\Big]\Vol\\+\int_M\Big[
\left |(\nabla^2f)_{[1][0]}\right |^2+ \left
|(\nabla^2f)_{[-1]}\right |^2-\frac {3}{2n}P(\gr)\Big]\Vol.
\end{multline}
\begin{rmrk}
Note that the expression in the  right-hand side of the first line in \eqref{e:bohinw2} is precisely the CR-Lichnerowicz positivity condition used to find a lower bound of the first eigenvalue of the sub-Laplacian  (see \cite{Gr,CC09a,CC09b,IVO}). Indeed, if one assumes the positivity condition \eqref{lich} to be in the form
$$Rc(X,X)+2(n+1)A(JX,X)\ge k_0g(X,X)$$
and $\triangle f =\lambda f$, one easily gets $\lambda\ge \frac n{n+1}k_0$ which is the Greenleaf's estimation of the first eigenvalue of the sub-Laplacian.
\end{rmrk}

\section{Proof of Theorem~\ref{main}}
We follow the approach of \cite{LT} in the Riemannian case and \cite{CSW} in the CR case. 

Denote by $\bar S$ the  average value of the scalar curvature,  $$\bar S=\int_MS\Vol.$$
Let  $\varphi$ be the unique solution of the following PDE:
\begin{equation}\label{scal}
\triangle\varphi=S-\bar{S}, \quad \int_M\varphi\Vol=0.
\end{equation}
We obtain 
\begin{equation}\label{in1}
\int_M(S-\bar S)^2\Vol=\int_M(S-\bar S)\triangle\varphi\Vol=\int_MdS(\g)\Vol,
\end{equation}
where we used an integration by parts to achieve the last equality.

We write, using \eqref{wric0}, the equality \eqref{rid} in the form
\begin{equation}\label{rid0}
Ric(X,Y)=Rc_0(X,Y)+\frac{S}{2n}g(X,Y)+2(n-1)A(JX,Y).
\end{equation}
In view of \eqref{rid0}, the second Bianchi identity \eqref{div} takes the form
$$2(\bi_{e_a}Rc_0)(e_a,X)+\frac1ndS(X)+4(n-1)(\bi_{e_a}A)(e_a,JX)=dS(X),
$$
which yields
\begin{equation}\label{bi2w}
dS(X)=\frac{2n}{n-1}(\bi_{e_a}Rc_0)(e_a,X)+4n(\bi_{e_a}A)(e_a,JX).
\end{equation}
Substitute \eqref{bi2w} into  \eqref{in1} to get after an integration by parts, applying \eqref{tortrace} and \eqref{wric}, the following relations
\begin{multline}\label{s1}
\int_M(S-\bar S)^2\Vol=\int_MdS(\g)\Vol\\=\frac{2n}{n-1}\int_M(\bi_{e_a}Rc_0)(e_a,\g)\Vol+4n\int_M(\bi_{e_a}A)(e_a,J\g)\Vol\\
=-\frac{2n}{n-1}\int_MRc_0(e_a,e_b)\bi^2\varphi(e_a,e_b)\Vol-4n\int_MA(e_a,Je_b)\bi^2\varphi(e_a,e_b)\Vol\\
=-\frac{2n}{n-1}\int_MRc_0(e_a,e_b)(\bi^2\varphi)_{[1][0]}(e_a,e_b)\Vol-4n\int_MA(e_a,Je_b)(\bi^2\varphi)_{[-1]}(e_a,e_b)\Vol.
\end{multline}
Applying the Young's inequality $2ab\le \alpha a^2+\frac1{\alpha}b^2$, we get from \eqref{s1} that
\begin{multline}\label{s2}
\int_M(S-\bar S)^2\Vol\\
\le\frac{n}{n-1}\int_M\Big[\alpha|Rc_0|^2+\frac1{\alpha}|(\bi^2\varphi)_{[1][0]}|^2\Big]\Vol-4n\int_MA(e_a,Je_b)(\bi^2\varphi)_{[-1]}(e_a,e_b)\Vol.
\end{multline}
At this point we need to evaluate the norm of the (1,1) trace-free part of the horizontal Hessian. To do this, we involve the positivity condition \eqref{cor}. We have
\begin{prop}
On a compact strictly pseudoconvex pseudohermitian manifold of dimension $2n+1$ for $n\ge2$ and for any smooth function $f$ we have
\begin{multline}  \label{e:bohin1wi}
\int_M(\triangle f)^2\Vol=\frac{2n}{2n+3}\int_M\Big[Rc(\nabla f,\nabla f)+2(n+2)A(J\nabla f,\nabla f)  \Big]\Vol \\+\int_M\Big[\frac{2n(n+3)}{(n-1)(2n+3)}\left |(\nabla^2f)_{[1][0]}\right |^2 +\frac{2n}{2n+3} \left
|(\nabla^2f)_{[-1]}\right |^2+\frac{4n^2}{2n+3}(df(\xi))^2\Big]\Vol.
\end{multline}
\end{prop}
\begin{proof}
It follows directly  from \eqref{e:bohinw1}, Lemma~\ref{l:GrLee} and \eqref{hes10}.
\end{proof}
The positivity condition \eqref{cor}
and \eqref{e:bohin1wi} yield the inequality
$$\int_M\left |(\nabla^2f)_{[1][0]}\right |^2\Vol\le\frac{(n-1)(2n+3)}{2n(n+3)}\int_M(\triangle f)^2\Vol,
$$
which taken with respect to the function $\varphi$, then substituted into \eqref{s2} and applying \eqref{scal}  lead to the inequalty
\begin{multline}\label{s2f}
\Big(1-\frac{2n+3}{2\alpha(n+3)}\Big)\int_M(S-\bar S)^2\Vol\\
\le\frac{\alpha n}{n-1}\int_M|Rc_0|^2\Vol-4n\int_MA(e_a,Je_b)(\bi^2\varphi)_{[-1]}(e_a,e_b)\Vol.
\end{multline}
Take $\alpha=\frac{2n+3}{n+3}$ to obtain from \eqref{s2f} the   inequality \eqref{s2fa}, which completes the proof of the first part of Theorem~\ref{main}.

Suppose we have  equality  in \eqref{s2fa}. Take $\alpha=\frac{2n+3}{n+3}$ into \eqref{s2}  and then use the expression for $\left |(\nabla^2f)_{[1][0]}\right |^2$ from  \eqref{e:bohin1wi} substituted into \eqref{s2} to get, taking into account \eqref{scal}, that
\begin{equation}\label{zero}0\le-\frac{n}{2n+3}\int_M\Big[ Rc(\nabla\varphi,\nabla\varphi)+2(n+2)A(J\nabla\varphi,\nabla\varphi)+ \left |(\nabla^2\varphi)_{[-1]}\right |^2+2n(d\varphi(\xi))^2\Big]\Vol.
\end{equation}
The positivity condition \eqref{cor} and \eqref{zero} imply
\begin{equation}\label{zer}
 Rc(\nabla\varphi,\nabla\varphi)+2(n+2)A(J\nabla\varphi,\nabla\varphi)=0;\quad (\nabla^2\varphi)_{[-1]}=0; \quad (d\varphi(\xi))^2=0.
\end{equation}
The equality in \eqref{s2fa} and the second equality in \eqref{zer} yield \eqref{seqa}.

Substitute \eqref{zer} into \eqref{e:bohin1wi} and use \eqref{seqa} to get
\begin{equation}\label{aab}\int_M\left |(\nabla^2\varphi)_{[1][0]}\right |^2\Vol=\frac{(n-1)(2n+3)}{2n(n+3)}\int_M(\triangle \varphi)^2\Vol=\frac{(2n+3)^2}{(n+3)^2}\int_M|Rc_0|^2\Vol.
\end{equation}
Applying \eqref{seqa} and \eqref{zer}  to \eqref{s1} yields
\begin{multline}\label{s11}
\int_M(S-\bar S)^2\Vol=\frac{2n(2n+3)}{(n-1)(n+3)}\int_M|Rc_0|^2\Vol\\
=-\frac{2n}{n-1}\int_MRc_0(e_a,e_b)(\bi^2\varphi)_{[1][0]}(e_a,e_b)\Vol,
\end{multline}
which together with \eqref{aab} implies that we have equality in the Young's inequality. This shows that
\begin{equation}
Rc_0(e_a,e_b)=-\frac{n+3}{2n+3}(\bi^2\varphi)_{[1][0]}(e_a,e_b),
\end{equation}
which implies that the contact form $\bar\theta=exp(-\frac{2n+3}{(n+3)(n+2)}\varphi)\theta$ will be pseudo-Einstein by \cite[Proposition~5.9]{DT}.
This completes the proof of Theorem~\ref{main}.
\subsection{Proof of Corollary~\ref{maincor}} Two integration by parts yield
\begin{multline*}
\int_MA(e_a,Je_b)(\bi^2\varphi)_{[-1]}(e_a,e_b)\Vol=\int_MA(e_a,Je_b)\bi^2\varphi(e_a,e_b)\Vol=-\int_M(\nabla_{e_a}A)(e_a,J\g)\Vol\\
=\int_M\varphi(\nabla_{e_b}\nabla_{e_a}A)(e_a,Je_b)\Vol=0
\end{multline*}
due to \eqref{co1}, which proves the first part.

To show the second part, suppose we have an equality in \eqref{s2fab0}. Then \eqref{zer} and \eqref{aab} hold true.  We obtain the following chain of equalities by integration by parts, using the second equality in \eqref{zer}:
\begin{multline}\label{azer}
\int_MA(J\g,\g)\Vol=-\int_M\varphi(\nabla_{e_a}A)(e_a,J\g)\Vol-\int_M\varphi A(e_a,Je_b)\nabla^2\varphi(e_a,e_b)\Vol\\=
-\int_M\varphi(\nabla_{e_a}A)(e_a,J\g)\Vol-\int_M\varphi A(e_a,Je_b)(\nabla^2\varphi)_{[-1]}(e_a,e_b)\Vol\\=-\int_M\varphi(\nabla_{e_a}A)(e_a,J\g)\Vol=\frac12\int_M\varphi^2(\nabla_{e_b}\nabla_{e_a}A)(e_a,Je_b)\Vol=0.
\end{multline}
Applying \eqref{azer}, Lemma~\ref{l:GrLee}  and the third equality in \eqref{zer} to \eqref{Aa}, we obtain  using \eqref{aab} that
\begin{multline}\label{fincor}
0=2\int_MA(\g,J\g)\Vol=\int_M\Big[\frac1{2n}(\triangle\varphi)^2-\frac1{n-1}\left |(\nabla^2\varphi)_{[1][0]}\right |^2\Big]\Vol\\=\int_M\Big[\frac1{2n}(\triangle\varphi)^2-\frac{2n+3}{2n(n+3)}(\triangle\varphi)^2\Big]\Vol=-\frac1{2(n+3)}\int_M(\triangle\varphi)^2\Vol.
\end{multline}
Now, \eqref{fincor} shows that $\triangle\varphi=0$ and  the manifold is pseudo-Einstein with constant pseudohermitian scalar curvature, which completes the proof of Corollary~\ref{maincor}.

The equality 
$$
\Big|Rc-\frac{\bar{S}}{2n}g\Big|^2=\Big|Rc-\frac{S}{2n}g\Big|^2+\frac1{2n}(S-\bar{S})^2
$$
together with Corollary~\ref{maincor} imply
\begin{cor}
In the conditions of Corollary~\ref{maincor} one has 
$$
\int_M\Big|Rc-\frac{\bar{S}}{2n}g\Big|^2\Vol\le\frac{n(n+4)}{(n-1)(n+3)}\int_M\Big|Rc-\frac{S}{2n}g\Big|^2\Vol.
$$

\end{cor}
\subsection{Proof of Corollary~\ref{mainsas}}
It remains to show only the last part of Corollary~\ref{mainsas}.  In the torsion-free case, i.e. in Sasakian case, it is well known that the Ricci tensor $Ric^h$ of the Riemannian metric $h=g+\eta\otimes\eta$ and the Webster Ricci tensor $Rc$ are connected by (see e.g. \cite{DT})
\begin{equation}\label{homo}Ric^h(X,X)=Rc(X,X)-2g(X,X), \qquad Ric^h(\xi,\xi)=2n.
\end{equation}
If we have pseudo-Einstein Sasakian structure, we obtain from \eqref{homo}
$$Ric^h(X,X)=\frac{S-4n}{2n}g(X,X), \qquad Ric^h=\frac{S-4n}{2n}h+\Big(2n-\frac{S-4n}{2n}\Big) \eta\otimes\eta
$$
showing that the Riemannian Sasaki structure is $\eta$-Einstein \cite{Okum} and the Tanno's D-homothetic deformation $\bar{\eta}=a\eta, \quad \bar{\xi}=\frac1a\xi, \quad \bar h=ah+a(a-1)\eta\otimes\eta$ makes it  a Sasaki-Einstein space for a suitable constant $a$ \cite{Tanno} {c.f. also \cite{BGM}.
\section{Proof of Corollary~\ref{im}}  For completeness and a better understanding of the proof of Corollary~\ref{im}, we give  a proof of \cite[Theorem~1.2]{CSW}, Theorem~\ref{thcw}.

The next result involves the positivity condition \eqref{lich}. We have
\begin{prop}
On a compact strictly pseudoconvex pseudohermitian manifold of dimension $2n+1$ for $n\ge2$ and for any smooth function $f$ we have
\begin{multline}\label{e:bohinw2i}
\int_M(\triangle f)^2\Vol=\frac{n}{n+1}\int_M\Big[Rc(\nabla f,\nabla f)+2(n+1)A(J\nabla f,\nabla f)\Big]\Vol\\+\int_M\Big[\frac{n(n+2)}{n^2-1}
\left |(\nabla^2f)_{[1][0]}\right |^2+ \frac{n}{n+1} \left
|(\nabla^2f)_{[-1]}\right |^2\Big]\Vol.
\end{multline}
\end{prop}
\begin{proof}
It follows directly  from \eqref{e:bohinw2} and  Lemma~\ref{l:GrLee}.
\end{proof}
The positivity condition \eqref{lich} and \eqref{e:bohinw2i} yield the inequality
$$\int_M\left |(\nabla^2f)_{[1][0]}\right |^2\Vol\le\frac{n^2-1}{n(n+2)}\int_M(\triangle f)^2\Vol,
$$
which taken with respect to the function $\varphi$, then substituted into \eqref{s2} and applying \eqref{scal}  lead to the inequalty
\begin{multline}\label{s2fi}
\Big(1-\frac{n+1}{\alpha(n+2)}\Big)\int_M(S-\bar S)^2\Vol
\le\frac{\alpha n}{n-1}\int_M|Rc_0|^2\Vol-4n\int_MA(e_a,Je_b)(\bi^2\varphi)_{[-1]}(e_a,e_b)\Vol.
\end{multline}
Taking $\alpha=\frac{2(n+1)}{n+2}$, we obtain from \eqref{s2fi} the   inequality \eqref{s2fb}, which completes the proof of the first part of Theorem~\ref{thcw} from \cite{CSW}.

The second part follows similarly to the proof of the second part of Theorem~\ref{main}. Indeed,
suppose we have  equality  in \eqref{s2fb}. Take $\alpha=\frac{2(n+1)}{n+2}$ into \eqref{s2}, then use the expression for $\left |(\nabla^2f)_{[1][0]}\right |^2$ from  \eqref{e:bohinw2i} substituted into \eqref{s2} to get, taking into account \eqref{scal}, that
\begin{equation}\label{zero1}0\le-\frac{n}{2(n+1)}\int_M\Big[ Rc(\nabla\varphi,\nabla\varphi)+2(n+1)A(J\nabla\varphi,\nabla\varphi)+ \left |(\nabla^2\varphi)_{[-1]}\right |^2\Big]\Vol.
\end{equation}
The positivity condition \eqref{lich} and \eqref{zero1} imply
\begin{equation}\label{zer1}
 Rc(\nabla\varphi,\nabla\varphi)+2(n+1)A(J\nabla\varphi,\nabla\varphi)=0;\quad (\nabla^2\varphi)_{[-1]}=0.
\end{equation}
The equality in \eqref{s2fb} and the second equality in \eqref{zer1} yield \eqref{seqb}.

Substitute \eqref{zer1} into \eqref{e:bohinw2i} and use \eqref{seqb} to get
\begin{equation}\label{aab1}\int_M\left |(\nabla^2\varphi)_{[1][0]}\right |^2\Vol=\frac{n^2-1}{n(n+2)}\int_M(\triangle\varphi)^2\Vol=\frac{(2n+2)^2}{(n+2)^2}\int_M|Rc_0|^2\Vol.
\end{equation}
Applying \eqref{seqb} and \eqref{zer1}  to \eqref{s1} yields
\begin{multline*}
\int_M(S-\bar S)^2\Vol=\frac{4n(n+1)}{(n-1)(n+2)}\int_M|Rc_0|^2\Vol
=-\frac{2n}{n-1}\int_MRc_0(e_a,e_b)(\bi^2\varphi)_{[1][0]}(e_a,e_b)\Vol,
\end{multline*}
which, together with \eqref{aab1}, implies that we have equality in the Young's inequality. This shows that
\begin{equation*}
Rc_0(e_a,e_b)=-\frac{n+2}{2n+2}(\bi^2\varphi)_{[1][0]}(e_a,e_b),
\end{equation*}
yielding  that the contact form $\bar\theta=exp(-\frac{1}{n+1}\varphi)\theta$ will be pseudo-Einstein by \cite[Proposition~5.9]{DT} which completes the proof of Theorem~\ref{thcw} of \cite{CSW}.

To finish the proof of  Corollary~\ref{im}, it remains to show only that if we have equality in \eqref{seqbc} then the manifold is pseudo-Einstein with constant pseudohermitian scalar curvature. 

Indeed, in the equality case we have that \eqref{zer1} and \eqref{aab1} are valid. The second identity in \eqref{zer1} implies that \eqref{azer} holds true. Apply \eqref{azer}, Lemma~\ref{l:GrLee}  and the second equality in \eqref{zer1} to \eqref{Aa} to get, using \eqref{aab1}, that
\begin{multline}\label{fincor1}
0=2\int_MA(\g,J\g)\Vol=\int_M\Big[ -\frac {1}{2n}g(\nabla^2 \varphi,\omega)^2+\frac1{2n}(\triangle\varphi)^2-\frac1{n-1}\left |(\nabla^2\varphi)_{[1][0]}\right |^2\Big]\Vol\\=\int_M\Big[ -\frac {1}{2n}g(\nabla^2 \varphi,\omega)^2+\frac1{2n}(\triangle\varphi)^2-\frac{n+1}{n(n+2)}(\triangle\varphi)^2\Big]\Vol\\=-\int_M\Big[\frac1{2(n+2)}(\triangle\varphi)^2 +\frac {1}{2n}g(\nabla^2 \varphi,\omega)^2\Big]\Vol.
\end{multline}
Now, \eqref{fincor1} shows that $\triangle\varphi=0$ and  the manifold is pseudo-Einstein with constant pseudohermitian scalar curvature which completes the proof.

\section{Appendix}
The purpose of this section is to record  for self-sufficiency  proofs of some of the results of \cite{Gr,L1,GL88} in real variables (see also \cite[Appendix]{IVO}) including the Greenleaf's CR Bochner formula \cite{Gr}, the CR Paneitz operator and its non-negativity for $n>1$ \cite{GL88}, etc.

\subsection{The Greenleaf's CR-Bochner formula}
\begin{thrm}\cite{Gr}
On a strictly pseudoconvex pseudohermitian manifold of dimension $2n+1$, $n\geq 1$, the following Bochner-type identity holds
\begin{equation}  \label{bohh1}
-\frac12\triangle |\nabla f|^2=-g(\nabla(\triangle f),\nabla f)+Ric(\nabla
f,\nabla f)+2A(J\nabla f,\nabla f) +|\nabla^2f|^2+ 4\nabla df(\xi,J\nabla f).
\end{equation}
\end{thrm}
\begin{proof}
By definition we have
\begin{equation}  \label{boh1}
-\frac12\triangle |\nabla f|^2=\nabla^3f(e_a,e_a,e_b)
df(e_b)+\nabla^2f(e_a,e_b)\nabla^2f(e_a,e_b) =\nabla^3f(e_a,e_a,e_b) df(e_b)
+ |\nabla^2f|^2.
\end{equation}
To evaluate the first term in the right-hand side of \eqref{boh1},
we use the Ricci identities \eqref{e:ricci identities}.
Taking into account  $(\nabla_X T)(Y,Z)=0$ and applying successively
the Ricci identities \eqref{e:ricci identities}, we obtain
\begin{equation}  \label{boh3}
\nabla^3f(e_a,e_a,e_b)df(e_b)= -g(\nabla(\triangle f),\nabla f )+Ric(\nabla
f,\nabla f)+2A(J\nabla f,\nabla f)+4\nabla^2f(\xi,J\nabla f).
\end{equation}
A substitution of \eqref{boh3} into \eqref{boh1} completes the proof of %
\eqref{bohh1}.
\end{proof}
The next integral formula was  originally proved in \cite{Gr}, we take it in real notations from \cite{IVO}.

\begin{lemma}\cite{Gr}
\label{gr2} On a compact strictly pseudoconvex pseudohermitian manifold
of dimension $2n+1$, $n\geq 1$, we have the identity
\begin{equation}  \label{2}
\int_M\nabla^2f(\xi,J\nabla f)\Vol=-\int_M\Big[2n(df(\xi))^2
+A(J\nabla f,\nabla f)\Big] \Vol.
\end{equation}
\end{lemma}
\begin{proof}
We consider the horizontal 1-form defined by
$
D(X)=df(JX)df(\xi),
$
whose divergence is, taking into account the second formula of
\eqref{e:ricci
identities},
\begin{equation}  \label{vert2}
(\bi_{e_a}D)(e_a)=\nabla^2f(e_a,Je_a)df(\xi)-\nabla^2f(\xi,J\nabla
f)-A(J\nabla f,\nabla f).
\end{equation}
Integrating \eqref{vert2} over $M$ and using \eqref{xi1} implies \eqref{2},
which completes the proof of the lemma.
\end{proof}

\subsection{The CR-Paneitz operator}
The famous CR-Paneitz operator is defined as follows \cite{Li,GL88}. 

Given a function $f$ we define the one form,
\begin{equation}  \label{e:Pdef1}
P(X)\equiv P_{f}(X)=\nabla ^{3}f(X,e_{b},e_{b})+\nabla
^{3}f(JX,e_{b},Je_{b})+4nA(X,J\nabla f),
\end{equation}%
and also  a fourth order differential operator  (the so called CR-Paneitz operator in \cite{Chi06}),
\begin{equation}  \label{e:Cdef1}
Cf=(\nabla_{e_a} P)({e_a})=\nabla ^{4}f(e_a,e_a,e_{b},e_{b})+\nabla
^{4}f(e_a,Je_a,e_{b},Je_{b})-4n\nabla^* A(J\nabla f)-4n\,g(\nabla^2 f,JA).
\end{equation}
According to \eqref{comp}, the horizontal Hessian $\bi^2f$ splits into two parts,
$\bi^2f=(\bi^2f)_{[1]} + (\bi^2f)_{[-1]}$ , where
\begin{equation}\label{comp1}
\begin{split} (\nabla ^{2}f)_{[1]}(X,Y)=\frac{1}{2}\left[ (\nabla
^{2}f)(X,Y)+(\nabla ^{2}f)(JX,JY)\right],\\  (\nabla ^{2}f)_{[-1]}(X,Y)=\frac{1}{2}\left[ (\nabla
^{2}f)(X,Y)-(\nabla ^{2}f)(JX,JY)\right].
\end{split}
\end{equation}
In view of \eqref{xi1}, the  trace-free part  $(\nabla ^{2}f)_{[1][0]}$ of the $(1,1)$ component of the horizontal
Hessian is given by
\begin{equation}\label{hes10}
\begin{split}
(\nabla ^{2}f)_{[1][0]}=(\nabla ^{2}f)_{[1]}+\frac{\triangle f}{2n} g(X,Y)+df(\xi)\,\omega (X,Y);\\
|(\nabla ^{2}f)_{[1][0]}|^2=|(\nabla ^{2}f)_{[1]}|^2-\frac{(\triangle f)^2}{2n}-2n(df(\xi))^2.
\end{split}
\end{equation}
One of the basic results relating the above defined operators is the next lemma proved in \cite{GL88}, which we present in real notations, see e.g. \cite{IVO}.
\begin{lemma}\cite{GL88}\label{l:GrLee}
On a compact strictly pseudoconvex pseudohermitian manifold
of dimension $2n+1$, $n\geq 1$, the following identities hold true
\begin{gather}\label{panz}
\nabla_{e_{a}}(\bi^2f)_{[1][0]}(e_{a},X) =\frac{n-1}{2n}P_f(X), \\\label{panz1}
\int_M |(\bi^2f)_{[1][0]}|^2\Vol=-\frac{n-1}{2n}\int_M P_f(\gr)\Vol=\frac{n-1}{2n}\int_M f(Cf)\Vol.
\end{gather}
In particular, if $n>1$ the CR-Paneitz operator is non-negative, $\int_M f(Cf)\Vol\ge 0$.
\end{lemma}
\begin{proof}
Taking into account the third and the fourth Ricci identity in  \eqref{e:ricci identities}, we have
\begin{equation*}
\begin{aligned}
\nabla ^{3}f(e_{a},e_{a},X)& =\nabla ^{3}f(X,e_{a},e_{a}) +Ric(X,\nabla
f)+4\nabla ^{2}f(\xi ,JX)+2A(JX,\gr ), \\
\nabla ^{3}f(e_{a},Je_{a},JX) &=\frac12\Big(\nabla ^{3}f(e_{a},Je_{a},JX)-\nabla ^{3}f(Je_{a},e_{a},JX)\Big)=-\rho(JX,\gr)-2n\bi^2f(\xi,JX).
\end{aligned}
\end{equation*}
The sum of the above equalities gives
\begin{multline}  \label{e:divB1}
2\nabla _{e_a}(\nabla ^{2}f)_{[1]}(e_{a},X)\\ =\nabla
^{3}f(X,e_{a},e_{a}) +Ric(X,\gr)-\rho(JX,\gr) +(4-2n)\bi^2f(\xi,JX)+2A(JX,\gr).
\end{multline}
The equality \eqref{xi1} yields  $\nabla^{2}f(JX,\xi )=-\frac {1}{2n}\nabla ^{3}f(JX,e_{a},Je_{a})$, which together with  the second Ricci  identity in \eqref{e:ricci identities} imply
\begin{equation}\label{xixi}\bi^2f(\xi,JX)=\nabla^{2}f(JX,\xi )-A(JX,\gr)=-\frac {1}{2n}\nabla ^{3}f(JX,e_{a},Je_{a})-A(JX,\gr).
\end{equation}
Therefore, using \eqref{rid} and \eqref{xixi}  we get from \eqref{e:divB1} that
\begin{equation}  \label{e:divB}
2\nabla _{e_a}(\nabla ^{2}f)_{[1]}(e_{a},X) =\nabla
^{3}f(X,e_{a},e_{a})+\frac{n-2}{n}\nabla
^{3}f(JX,e_{a},Je_{a}) 
+4(n-1)A(X,J\nabla f).
\end{equation}
The divergence of the trace part of $(\bi^2f)_{[1]}$ is computed as follows
\begin{multline}  \label{e:divtrB}
\nabla _{e_a}\left(- \frac{1}{2n}\triangle f\cdot g-df(\xi )\omega
\right) (e_{a},X) =\frac{1}{2n}\nabla ^{3}f(X,e_{a},e_{a})+\nabla
^{2}f(JX,\xi ) \\
=\frac{1}{2n}\nabla ^{3}f(X,e_{a},e_{a})-\frac{1}{2n}\nabla
^{3}f(JX,e_{a},Je_{a}).
\end{multline}
Now, the identities \eqref{e:divB} and \eqref{e:divtrB} imply  \eqref{panz}. 
The second identity \eqref{panz1} follows by an integration by parts from \eqref{panz}.
\end{proof}
The next result, essentially proved in \cite{CC09a}, involves the CR-Paneitz operator. We present it in real notations from \cite[Lemma~8.7]{IVO}:
\begin{lemma}\label{grn3}
On a strictly pseudoconvex pseudohermitian manifold
of dimension $2n+1$, $n\geq 1$, we have the identity
$$
\nabla ^{2}f(\xi ,Z)=\frac{1}{2n}\nabla ^{3}f(Z,Je_a,e_a)-A(Z,\nabla f).
$$ 
Additionally, in the compact case, the following integral identity holds true:
\begin{equation}\label{idcr2}
\int_{M}\nabla ^{2}f(\xi ,J\nabla f)Vol_{\theta }=\int_{M}\Big [-\frac{1}{2n}\left(
\triangle f\right) ^{2}+A(J\nabla f,\nabla f)-\frac{1}{2n}P_f(\gr)\Big ]\Vol.
\end{equation}
\end{lemma}
\begin{proof}
We compute using the first two Ricci identities in \eqref{e:ricci identities}
\begin{multline*}
2\nabla ^{3}f(Z,Je_a,e_a)=\nabla ^{3}f(Z,Je_a,e_a)-\nabla ^{3}f(Z,e_a,Je_a)=-2\omega(Je_a,e_a)\nabla^2 f(Z,\xi)\\=4n\left (  \nabla^2f(\xi,Z) +A(Z,\gr)\right),
\end{multline*}
which proves the first formula.  The second identity \eqref{idcr2} follows from the first formula, the definition \eqref{e:Pdef1} of $P_f$ and an integration by parts.
\end{proof}
Combining \eqref{idcr2} with \eqref{2}, one gets
\begin{equation}\label{Aa}
2\int_M  A(J\nabla f,\nabla f)\Vol =\int_M\Big[ -\frac {1}{2n}g(\nabla^2 f,\omega)^2+\frac {1}{2n}(\triangle f)^2+\frac {1}{2n}P_f(\gr)\Big]\Vol.
\end{equation}
Integrating the Bochner type formula \eqref{bohh1}  on a compact $M$ gives
\begin{multline}  \label{bohin1}
0=\int_M\Big[-(\triangle f)^2+\left |(\nabla^2f)_{[1]}\right |^2 + \left
|(\nabla^2f)_{[-1]}\right |^2\\+Ric(\nabla f,\nabla f)+2A(J\nabla f,\nabla f)
+ 4\nabla^2f(\xi,J\nabla f)\Big] \Vol.
\end{multline}
We use Lemma \ref{grn3} to represent the last term, which turns the  identity \eqref{bohin1} into the following 
\begin{multline}  \label{e:bohin}
0=\int_M\Big[-(\triangle f)^2+\left |(\nabla^2f)_{[1]}\right |^2 + \left
|(\nabla^2f)_{[-1]}\right |^2+Ric(\nabla f,\nabla f)+6A(J\nabla f,\nabla f)\\
-\frac {2}{n}(\triangle f)^2 -\frac {2}{n}P_f(\gr)\Big]\Vol.
\end{multline}
After a substitution of \eqref{Aa}  in \eqref{e:bohin}, taking into account \eqref{hes10}, we get
\begin{multline}\label{e:bohin-best}
0=\int_M\Big[Ric(\nabla f,\nabla f)+4A(J\nabla f,\nabla f)-\frac {n+1}{n}(\triangle f)^2\Big]\Vol\\
+\int_M\Big[\left |(\nabla^2f)_{[1][0]}\right |^2+ \left
|(\nabla^2f)_{[-1]}\right |^2-\frac {3}{2n}P_f(\gr)\Big]\Vol.
\end{multline}
The above identities  \eqref{e:bohin} and \eqref{e:bohin-best} are  the key identities relating the CR-Paneitz operator and the Greenleaf's CR-Bochner formula \eqref{bohin1} on a compact manifold, and were  essentially proved in \cite{CC09a}. We presented these identities  in real notations from   \cite[(8.13), (8.15)]{IVO}.

\end{document}